\documentclass{amsart}

\usepackage{amsmath,amssymb,amsthm,mathrsfs}
\usepackage{color}

\newcommand{\BC}{{\mathbb C}}\newcommand{\BD}{{\mathbb D}}

\newcommand{\BR}{{\mathbb R}}
\newcommand{\BT}{{\mathbb T}}

\newcommand{\cD}{{\mathcal D}}

\newcommand{\cH}{{\mathcal H}}

\newcommand{\cK}{{\mathcal K}}\newcommand{\cL}{{\mathcal L}}

\newcommand{\cU}{{\mathcal U}}\newcommand{\cV}{{\mathcal V}}
\newcommand{\cW}{{\mathcal W}}

\newcommand{\fC}{{\mathfrak C}}\newcommand{\fD}{{\mathfrak D}}

\newcommand{\fI}{{\mathfrak I}}

\newcommand{\fR}{{\mathfrak R}}

\newcommand{\wtilF}{\widetilde{F}}
\newcommand{\wtilG}{\widetilde{G}}

\newcommand{\wtilM}{\widetilde{M}}

\newcommand{\wtilQ}{\widetilde{Q}}\newcommand{\wtilR}{\widetilde{R}}

\newcommand{\wtilW}{\widetilde{W}}\newcommand{\wtilX}{\widetilde{X}}
\newcommand{\wtilY}{\widetilde{Y}}

\newcommand{\whatJ}{\widehat{J}}

\newcommand{\be}{\beta}
\newcommand{\ga}{\gamma}
\newcommand{\de}{\delta}
\newcommand{\vep}{\varepsilon}

\newcommand{\la}{\lambda}

\newcommand{\cran}{\textup{R}\overline{\textup{an}}\,}
\newcommand{\ran}{\textup{Ran\,}}

\newcommand{\im}{\textup{Im\,}}
\newcommand{\re}{\textup{Re\,}}
\newcommand{\kr}{\textup{Ker\,}}

\newcommand{\mat}[2]{\ensuremath{\left[\begin{array}{#1}#2\end{array} \right]}}

\newcommand{\sbm}[1]{\left[\begin{smallmatrix} #1\end{smallmatrix}\right]}
\newcommand{\ov}[1]{{\overline{#1}}}

\newcommand{\inn}[2]{\ensuremath{\langle #1,#2 \rangle}}

\newcommand{\wtil}[1]{{\widetilde{#1}}}

\newcommand{\half}{\frac{1}{2}}
\newcommand{\ands}{\quad\mbox{and}\quad}

\newtheorem{theorem}{Theorem}[section]
\newtheorem{corollary}[theorem]{Corollary}
\newtheorem{lemma}[theorem]{Lemma}
\newtheorem{proposition}[theorem]{Proposition}

\newtheorem{example}[theorem]{Example}

\newcommand{\PR}{\textup{PR}}
\newcommand{\OPR}{\textup{PR}_\circ}

\newcommand{\precc}{\prec_\fC}
\newcommand{\simc}{\sim_\fC}

\begin{document}

\title[A pre-order on positive real operators]{A pre-order on positive real operators and its invariance under linear fractional transformations}
\author{S. ter Horst}


\date{}

\maketitle

\begin{abstract}
A pre-order and equivalence relation on the class of positive real Hilbert space operators are introduced, in correspondence with similar relations for contraction operators defined by Yu.L. Shmul'yan in \cite{S80}. It is shown that the pre-order, and hence the equivalence relation, are preserved by certain linear fractional transformations. As an application, the operator relations are extended to the class $\fC(\cU)$ of Carath\'eodory functions on the unit disc $\BD$ of $\BC$ whose values are operators on a finite dimensional Hilbert space $\cU$. With respect to these relations on $\fC(\cU)$ it turns out that the associated linear fractional transformations of $\fC(\cU)$ preserve the equivalence relation on their natural domain of definition, but not necessarily the pre-order, paralleling similar results for Schur class functions in \cite{tHsub}.
\end{abstract}

\setcounter{section}{-1}
\section{Introduction}
\setcounter{equation}{0}

In this paper we introduce a pre-order on the set of positive real Hilbert space operators and prove that this pre-order, and its associated equivalence relation, are preserved by linear fractional transformations of the type extensively studied by V.M. Potapov, cf., \cite{AD08} and the references therein. The pre-order is similar to one defined by Yu.L. Shmul'yan in \cite{S80} for contractions, which was recently extended in \cite{tHsub} to Schur class functions. Due to certain properties of the map $A\mapsto \re(A):=\half(A+A^*)$ the proofs are more transparent and a more complete characterization of the equivalence relation is obtained. As an application of our results, we can (partially) extend our results to the class of Carath\'eodory functions, paralleling the main results of \cite{tHsub}, directly at the level of functions and without considering Toeplitz operators.

In order to state our results more precisely, we require some preliminaries. Let $\cH$ and $\cK$ be Hilbert spaces. We write $\cL(\cH,\cK)$ for the set of operators mapping $\cH$ into $\cK$. If $\cH=\cK$ we abbreviate $\cL(\cH,\cK)$ to $\cL(\cH)$. Here ``operator on $\cH$'' means a bounded linear map. Moreover, invertibility of a Hilbert space operator will always mean boundedly invertible. With $\PR(\cH)$ we indicate the set of positive real operators on $\cH$, that is, the operators $A\in\cL(\cH)$ with real part $\re(A)$ a positive operator, notation $\re(A)\geq0$. The imaginary part $\frac{1}{2i}(A-A^*)$ of $A$ is denoted by $\im(A)$ and the support subspaces of $\re(A)$ and $\im(A)$ are denoted by $\fR_A$ and $\fI_A$, respectively. Further, the set of invertible operators in $\PR(\cH)$ is denoted by $\OPR(\cH)$, equivalently, these are the operators on $\cH$ whose real parts are strictly positive, notation $\re(A)>0$.

Given $A,B\in\PR(\cH)$, we write $A\prec B$ if
\begin{align}\label{PreorderIntro}
A-B= \re(B)^{\half}X \re(B)^{\half}\quad \mbox{ for some }\quad X\in\fR_B.
\end{align}
In Theorem \ref{T:prec} below it is proved that $\prec$ defines a pre-order on $\PR(\cH)$ and several reformulations of $A\prec B$ are given. According to Lemma \ref{L:defrel1} below,
The relation $A\prec B$ implies the range-inclusion $\ran\re(A)\subset\ran\re(B)$ and thus $\fR_A\subset\fR_B$. Similarly, Theorem \ref{T:eqrel} provides a few characterizations of the associated equivalence relation, denoted by $\sim$. In particular, it is shown that $A\sim B$ holds if and only if $\fR_A=\fR_B$ and
\begin{align}\label{EquivIntro}
A-B= \re(A)^{\half}\wtilX \re(B)^{\half}\quad \mbox{ for some }\quad \wtilX\in\cL(\fR_B).
\end{align}
Since $\re$ has the following properties:
\[
\re(A+B)=\re(A)+\re(B)\ands \re(A^*)=\re(A),
\]
the proofs for the positive real case are simpler and more transparent than for the case of contraction operators, leading to a more complete characterization for the equivalence relation than the one obtained in \cite{KSS91}, cf., \cite[Thoerem 1.6]{tHsub}.

Now define $J$ and $\whatJ$ to be the signature matrices in $\cL(\cH\oplus\cH)$ given by 
\begin{align}\label{JJhat}
J=\mat{cc}{0&-I\\-I&0}\ands \whatJ=\mat{cc}{I&0\\0&-I}.
\end{align}
Let $W\in \cL(\cH\oplus\cH)$ be $J$-contractive, that is,
\begin{align}\label{Wcond}
W^*J W\leq J.
\end{align}
In terms of the entries of the $2\times 2$ block decomposition $W=\sbm{W_{11}&W_{12}\\W_{21}&W_{22}}$ of $W$ this means
\begin{align*}
\mat{cc}{2\re(W_{11}^*W_{21}) & W_{21}^*W_{12}+W_{11}^*W_{22}-I\\
W_{12}^*W_{21}+W_{22}^*W_{11}-I & 2\re(W_{12}^*W_{22})}\geq 0.
\end{align*}
With $W$ we associate the linear fractional transformation (LFT) $T_W$ defined by
\begin{align}\label{LFT}
T_W[A]:=(W_{11}A+W_{12})(W_{21}A+W_{22})^{-1}\quad (A\in\fD_W).
\end{align}
Here $\fD_W:=\{A\in \PR(\cH)\, \colon\, \mbox{$W_{21}A+W_{22}$ invertible}\}$ is the domain of $T_W$. The assumption \eqref{Wcond} implies that $\OPR(\cH)\subset \fD_W$, in particular, $W_{22}$ is invertible, and that $T_W$ maps $\fD_W$ into $\PR(\cH)$. For more details we refer to \cite[Section 2.10]{AD08}; one easily verifies that the `finite dimensional' algebraic results proved there extend to general Hilbert spaces.

Assume $W$ is invertible. Following Section 2.3 in \cite{AD08} we define $\wtilW\in \cL(\cH\oplus\cH)$ by
\begin{align}\label{Wtil}
\wtilW=\whatJ W^{-1}\whatJ
=:\mat{cc}{\wtilW_{11}&\wtilW_{12}\\\wtilW_{21}&\wtilW_{22}}.
\end{align}
Then $\fD_W=\{A\in \PR(\cH) \colon \mbox{$\wtilW_{11}+A\wtilW_{21}$ invertible}\}$ and $T_W$ can be written as
\begin{align}\label{AltLFT}
T_W[A]=(\wtilW_{11}+A\wtilW_{21})^{-1}(\wtilW_{12}+A\wtilW_{22})\quad (A\in\fD_W).
\end{align}

Our first main result is the following theorem.

\begin{theorem}\label{T:MainPrec}
Let $W\in \cL(\cH\oplus\cH)$ be invertible and assume \eqref{Wcond} is satisfied. Then $T_W$ preserves the pre-order $\prec$ on $\PR(\cH)$ restricted to $\fD_W$.
\end{theorem}

Consequently, we obtain the following result.

\begin{theorem}\label{T:MainEquiv}
Let $W\in \cL(\cH\oplus\cH)$ be invertible and assume \eqref{Wcond} is satisfied. Then $T_W$ preserves the equivalence relation $\sim$ on $\PR(\cH)$ restricted to $\fD_W$.
\end{theorem}

Both theorems will be proved in Section \ref{S:LFT}, in an extended form. It turns out that Theorem \ref{T:MainEquiv} is considerably easier to prove and doing so leads to an observation that will be of use in the sequel: for $A,B\in \PR(\cH)$ such that $A\sim B$, with $\wtilX$ as in \eqref{EquivIntro}, the operator that established the equivalence $T_W[A]\sim T_W[B]$ (as in \eqref{EquivIntro} but with $\wtilX_W$ instead of $X$) satisfies $\|\wtilX_W\|\leq \|\wtilX\|$.

Now let $\cU$ be a finite dimensional Hilbert space. We write $\fC(\cU)$ for the Carath\'eodory class consisting of holomorphic $\PR(\cU)$-valued functions on the open unit disc $\BD=\{\la\in\BC\,\colon\, |\la|<1\}$. The strict Carath\'eodory class is denoted by $\fC_\circ(\cU)$ and consists of functions $F\in \fC(\cU)$ such that there exists a $\rho>0$ with $\re(F)(\la)\geq\rho I$ for each $\la\in\BD$.

Let $F\in\fC(\cU)$. Consider Theorem 5 from \cite{S76} with $J$ as defined above and $T(\la)=\sbm{F(\la)&I\\I&0}$, taking $\la=\la_0=:\la_1$ and $\mu=\mu_0=:\la_2$. This yields the existence of a function $\Phi:\BD^2\to\cL(\cU)$ such that
\begin{align*}
F(\la_1)-F(\la_2)=\re(F)(\la_1)\Phi(\la_1,\la_2)\re(F)(\la_2)\quad (\la_1,\la_2\in\BD).
\end{align*}
In view of \eqref{EquivIntro}, this means that $F(\la_1)\sim F(\la_2)$ for any two points $\la_1,\la_2\in\BD$, and thus, by Lemma \ref{L:defrel1} below, that $\fR_{F(\la)}$ is independent of the choice of $\la\in\BD$. Hence, we can define
\[
\fR_F:=\fR_{F(\la)}\quad \mbox{with $\la\in\BD$ arbitrary.}
\]
The function $F$ can be extended a.e.\ to a $\PR(\cU)$-valued function on the unit circle $\BT:=\{\la\in\BC\,\colon\, |\la|=1\}$ by taking non-tangential limits. On the boundary $\BT$ we have the inclusion $\fR_F\subset \fR_{F(\tau)}$, for a.e.\ $\tau\in\BT$, by the maximum principle, but in general not the reversed inclusion.

Since for any function in $\fC(\cU)$ the values on $\BD$ are all pairwise equivalent, for $F,G\in\fC(\cU)$ to satisfy $F(\la)\sim G(\la)$ for all $\la\in\BD$ it suffices to verify similarity of $F$ and $G$ at any one point of $\BD$. Hence the functions $\la\mapsto 1$ and $\la\mapsto (1+\la)(1-\la)^{-1}$ are pointwise equivalent on $\BD$. However, they have very different boundary behavior at $\la=1$. It turns out that the more natural extension of the pre-order and equivalence relation from Section \ref{S:pre-eq} to the Carath\'eodory class $\fC(\cU)$ involves an additional uniformity constraint. For $F,G\in\fC(\cU)$ we write $F\precc G$ if
\begin{equation}\label{CaraPre}
\begin{aligned}
&F-G= \re(G)^{\half}Q \re(G)^{\half},&\\
&\mbox{with $Q$ a bounded $\cL(\fR_G)$-valued function on $\BD$.}&
\end{aligned}
\end{equation}
Various characterizations of this pre-order, and the associated equivalence relation (denoted $\simc$), are proved in Theorem \ref{T:preceqC} below. In particular, $F\precc G$ is equivalent to $G-\vep(F-G)$ being in $\fC(\cU)$ for $\vep\in\BC$ sufficiently small. 
The most interesting implications of $F\precc G$ are with respect to the boundary behavior of $F$ and $G$. Proposition \ref{P:limbehavior} below shows that for any $u\in\cU$ and $\be\in\BT$, $\lim_{\la\to \be} G(\la)u=0$ implies $\lim_{\la\to \be} F(\la)u=0$ ( with convergence either both nontangentially or both unrestrictedly) provided $F\precc G$ and $G(\be)$ exists in $\PR(\cU)$.

Next we consider how Theorems \ref{T:MainPrec} and \ref{T:MainEquiv} can be extended to $\fC(\cU)$. Let $\Psi=\sbm{\Psi_{11}&\Psi_{12}\\ \Psi_{21}&\Psi_{22}}$ be a $\cL(\cU\oplus\cU)$-valued holomorphic function on $\BD$ such that
\begin{align}\label{Psicond}
\Psi(\la)^*J\Psi(\la)\leq J\qquad (\la\in\BD).
\end{align}
Then we can define the linear fractional transformation $T_\Psi$ by
\begin{align*}
T_\Psi[F]:=(\Psi_{11}+\Psi_{12}F)(\Psi_{22}+\Psi_{21}F)^{-1}\quad (F\in\fD_\Psi).
\end{align*}
Here all operations are pointwise and
\[
\fD_\Psi:=\{F\in \fC(\cU)\, \colon\, \mbox{$\Psi_{22}+\Psi_{21}F$ is invertible at each point of $\BD$}\}
\]
is the domain of $T_W$. Condition \eqref{Psicond} implies $\fC_\circ(\cU)\subset\fD_\Psi$ and $T_\Psi$ maps $\fD_\Psi$ into $\fC(\cU)$. In particular, $\Psi_{22}$ is invertible at each point of $\BD$. In order to extend Theorem \ref{T:MainEquiv} to $\fC(\cU)$, it suffices to assume that $\det(\Psi)\not\equiv 0$, which, due to the analyticity of $\Psi$, is equivalent to $\Psi$ being invertible at all except for a few isolated points of $\BD$. This condition is met in case $\Psi$ is $J$-unitary (see \cite[Section 4.1]{AD08}), i.e., if in addition \eqref{Psicond} holds with equality a.e.\ on $\BT$.

\begin{theorem}\label{T:MainCaraEquiv}
Let $\Psi$ be an $\cL(\cU\oplus\cU)$-valued holomorphic function on $\BD$ such that \eqref{Psicond} holds. Assume $\det(\Psi)\not\equiv0$. Then $T_\Psi$ preserves the equivalence relation $\simc$ on $\fC(\cU)$ restricted to $\fD_\Psi$.
\end{theorem}

Theorem \ref{T:MainPrec} does not extend to $\fC(\cU)$ on the domain $\fD_\Psi$. See Example \ref{E:ce} below. One has to reduce the domain to a smaller set. As before, assume $\det(\Psi)\not\equiv0$. Write $\BD_\Psi$ for the subset of $\BD$ where $\Psi$ is invertible. Define $\wtil{\Psi}:=\whatJ\Psi^{-1}\whatJ$ on $\BD_\Psi$. Then for all $F\in\cD_\Psi$ we have
\[
T_\Psi[F]=(\wtil{\Psi}_{11}+A\wtil{\Psi}_{21})^{-1}(\wtil{\Psi}_{12}+A\wtil{\Psi}_{22})\quad \mbox{on}\quad \BD_\Psi,
\]
with $\wtil{\Psi}_{ij}$ the entries from the standard $2\times 2$ block decomposition of $\wtil{\Psi}$. Write $H^\infty(\cU)$ for the Hardy class of bounded holomorphic $\cL(\cU)$-valued functions on $\BD$. We now define the reduced domain for $T_\Psi$ by
\[
\fD_\Psi^\circ=\{F\in\fC(\cU)\cap H^\infty(\cU)\,\colon\, (\wtil{\Psi}_{11}+F\wtil{\Psi}_{21})^{-1} \mbox{ exists and is bounded on } \BD_\Psi\}.
\]

\begin{theorem}\label{T:MainCaraPrec}
Let $\Psi$ be an $\cL(\cU\oplus\cU)$-valued holomorphic function on $\BD$ such that \eqref{Psicond} holds. Assume $\det(\Psi)\not\equiv0$. Then $T_\Psi$ preserves the pre-order $\precc$ on $\fC(\cU)$ restricted to $\fD_\Psi^\circ$.
\end{theorem}

Besides the current introduction, the paper consists of three sections. In section \ref{S:pre-eq} we prove that \eqref{PreorderIntro} defines a pre-order on $\PR(\cH)$ and derive various reformulations of \eqref{PreorderIntro}, and the associated equivalence relation, as well as several implications of these relation. Theorems \ref{T:MainPrec} and \ref{T:MainEquiv} will be proved in Section \ref{S:LFT}, in a slightly extended form. The extensions of the results from Sections \ref{S:pre-eq} and \ref{S:LFT} to the Carath\'eodory class are the topic of Section \ref{S:Cara}.

\section{A pre-order and equivalence relation on $\PR(\cH)$}\label{S:pre-eq}
\setcounter{equation}{0}

Throughout this section, let $\cH$ be a Hilbert space. The first result of this section shows that the relation on $\PR(\cH)$ given by \eqref{PreorderIntro} defines a pre-order $\prec$, and provide a few additional characterizations of $\prec$.

\begin{theorem}\label{T:prec}
The relation $A\prec B$ defined by one of the following four equivalent conditions:
\begin{itemize}
\item[(POi)]  $A-B= \re(B)^{\half}X \re(B)^{\half}$  for some $X\in\cL(\fR_B)$;

\item[(POii)] $A^*+B= \re(B)^{\half}Y \re( B)^{\half}$ for some $Y\in\cL(\fR_B)$;

\item[(POiii)] there exist $r>0$ with $B+\vep(A-B)\in\PR(\cH)$ for all $\vep\in\BC$ with $|\vep|\leq r$;

\item[(POiv)] there exist $r>0$ with $B+\vep(A-B)\in\PR(\cH)$ for all $\vep\in\BC$ with $|\vep|= r$;

\end{itemize}
defines a pre-order relation on $\PR(\cH)$. Assume $A \prec B$.  Then $X$ in (POi) and $Y$ in (POii) satisfy
\begin{equation}\label{XYrel}
Y-X^*=2I\ands
\re(X)+I=\re(Y)-I=X+Y\geq 0.
\end{equation}
Moreover, $r$ in (POiii)--(POiv) can be chosen such that $\|X\|$, $\|Y\|$ and $r$ satisfy
\begin{align}\label{XYrIneqs}
\|X\|\leq 2+\|Y\|,\quad \|Y\|\leq 2+\|X\|,\quad
\frac{1}{r}\leq \|X\|\leq \frac{2}{r}.
\end{align}
\end{theorem}

Before proving Theorem \ref{T:prec}, first observe that the following useful identities
\begin{equation}\label{ReAReB}
2\re(A)=(A-B)+(A^*+B)\ands
2\re(B)=(A^*+B)^*-(A-B),
\end{equation}
holds for any $A,B\in\cL(\cH)$. It is convenient to first prove the next lemma.

\begin{lemma}\label{L:defrel1}
Let $A,B\in\PR(\cH)$ satisfy (POi) and (POii). Then $\ran \re(A)$ is included in $\ran \re(B)$, and thus $\fR_A\subset\fR_B$, and there exists an operator $M\in\cL(\fR_B)$ with $M^*M=\half(X+Y)$ such that $\re(A)^\half=M\re(B)^\half$.
\end{lemma}

\begin{proof}[\bf Proof.]
The first identity in \eqref{ReAReB} yields $2\re(A)=\re(B)^\half(X+Y)\re(B)^\half$. Since $\re(A)\geq0$, we have $X+Y\geq0$. Therefore, Douglas' lemma \cite{D66} implies that there exists a co-isometry $N\in\cL(\fR_B,\fR_A)$ with $\kr N=\ran (X+Y)^\perp$ such that $\sqrt{2}\re(A)^{\half}=N(X+Y)^\half\re(B)^\half$. Hence we can take $M=2^{-1/2} N(X+Y)^{\half}$. Since $\ran (X+Y)^\half$ is included in the support of $N$, this shows $M^*M=\half(X+Y)^\half N^*N(X+Y)^\half=\half(X+Y)$. The inclusion $\ran \re(A)\subset\ran \re(B)$ now follows from $\re(A)^\half=\re(B)^\half M^*$.
\end{proof}

\begin{proof}[\bf Proof of Theorem \ref{T:prec}.]
We first proof the equivalence of (POi)--(POiv).

{\bf (POi) $\Longleftrightarrow$ (POii):}
The equivalence of (POi) and (POii) follows immediately from the second identity in \eqref{ReAReB}, with $X$ and $Y$ related through the identity given in \eqref{XYrel}.

{\bf (POi) $\Longleftrightarrow$ (POiii):}
Assume $X\in \cL(\fR_B)$ such that (POi) holds. Then, for any $\vep \in\BC$
\begin{align*}
\re (B+\vep(A-B))
&=\re(B) +\re (\vep (\re(B)^\half X \re(B)^\half)\\
&=\re(B) +\re(B)^\half \re(\vep X)\re(B)^\half\\
&=\re(B)^\half(I+ \re(\vep X))\re(B)^\half.
\end{align*}
Now, if $|\vep|\leq 1/\|X\|$ (with no limitation on $\vep$ if $X=0$), then
\begin{align*}
I+ \re(\vep X)\geq (1-\|\re(\vep X)\|)I\geq (1-\|\vep X\|)I=(1-|\vep| \|X\|)I\geq 0.
\end{align*}
Hence $B+\vep(A-B)\in\PR(\cH)$ whenever $|\vep|\leq 1/\|X\|$, which shows (POiii) holds with $r=1/\|X\|$.

Conversely, assume $r>0$ such that (POiii) holds. Take $\vep=\pm r$. Then (POiii) yields
\[
-\frac{1}{r}\re(B)\leq\re(A-B)\leq \frac{1}{r}\re(B).
\]
Next take $\vep=\pm i r$. Then $\re(\vep (A-B))=\mp r\im(A-B)$, and (POiii) yields
\[
-\frac{1}{r}\re(B)\leq\im(A-B)\leq \frac{1}{r}\re(B).
\]
By Lemma 1.4 in \cite{tHsub}, there exist self-adjoint operators $X_1$ and $X_2$ in $\cL(\fR_B)$ with $\|X_j\|= 1/r$, $j=1,2$, such that
\[
\re(A-B)=\re(B)^\half X_1 \re(B)^\half \ands \im(A-B)=\re(B)^\half X_2 \re(B)^\half.
\]
 Thus (POi) holds with $X=X_1+iX_2$ and we have $\|X\|\leq \|X_1\|+\|X_2\|=2/r$.

{\bf (POiii) $\Longleftrightarrow$ (POvi):}
The implication (POiii) $\Rightarrow$ (POiv) is obvious. Conversely, assuming (POiv) holds. Let $T_j=B+\vep_j(A-B)$ for $|\vep_j|=r$, $j=1,2$. If $T_1$ and $T_2$ are in $\PR(\cH)$, than so is $\half (T_1+T_2)=B-\half(\vep_1+\vep_2)(A-B)$. Now (POiii) follows because
\[
\left\{\frac{\vep_1+\vep_2}{2}\colon |\vep_1|=|\vep_2|=r\right\}
=\frac{r}{2}(\BT+\BT)=r\ov{\BD}=\{z\colon |z|\leq r\}.
\]
In particular, for $r$ in (POiii) we can take the same $r$ as in (POiv).

Clearly $A\prec A$ for any $A\in\PR(\cH)$; simply take $X=0$, $Y=I$ or any $r>0$. Hence, to see that $\prec$ defines a pre-order, it remains to show $\prec$ is transitive. Assume $A,B,C\in\PR(\cH)$ such that $A\prec B$ and $B\prec C$, say the relations are established through (POi) via $X_1\in\cL(\fR_B)$ and $X_2\in\cL(\fR_C)$, respectively. By Lemma \ref{L:defrel1}, $\re(B)^\half=M \re(C)^\half$ for some $M\in\cL(\cH)$. Hence
\[
A-C=A-B+B-C=\re(C)^\half(M^*X_1M+X_2)\re(C)^\half.
\]
Thus $A\prec C$, and we obtain that $\prec$ is transitive.

The identity $Y-X^*=2I$ and the positivity of $X+Y$, by Lemma \ref{L:defrel1}, show
\[
2(\re(Y)-I)=2(\re(X)+I)=X+Y\geq0.
\]
Hence the inequalities of \eqref{XYrel} hold as well.

The inequalities of \eqref{XYrIneqs} follow directly from the relations between $X$ and $Y$ and between $\|X\|$ and $r$ derived above.
\end{proof}

It now follows immediately from the various characterizations in Theorem \ref{T:prec} that $A\prec B$ implies:
\begin{itemize}
\item[(i)]
$A^*\prec B^*$;

\item[(ii)] $\fR_A\subset \fR_B$ and thus $\kr \re(B)\subset \kr \re(A)$;

\item[(iii)]
$i\im(A)|_{\fR_B^\perp}=A|_{\fR_B^\perp}=B|_{\fR_B^\perp}=i\im(B)|_{\fR_B^\perp}= -A^*|_{\fR_B^\perp}=-B^*|_{\fR_B^\perp}$;

\item[(iv)] $C^*AC\prec C^*BC$ for any $C\in\cL(\cH',\cH)$.

\end{itemize}

The equivalence relation associated with the pre-order $\prec$ will be indicated by $\sim$. Hence $A\sim B$ holds if and only if $A\prec B$ and $B\prec A$. If $A\sim B$, then the conclusion of Lemma \ref{L:defrel1} can be extended in the following way. Here and in the sequel, for an invertible operator $C$, the notation $C^{-*}$ indicates the operator $(C^{-1})^*$.

\begin{lemma}\label{L:defrel2}
Assume $A\sim B$. Let $X,Y\in \cL(\fR_B)$ be as in (POi) and (POii) and let $X',Y'\in \cL(\fR_A)$ be the operators associated with (POi) and (POii), respectively, for $B\prec A$. Then $\ran \re(A)=\ran \re(B)$, and thus $\fR_A=\fR_B$, and $\re(A)^\half=M\re(B)^\half$ holds for an invertible operator $M\in\cL(\fR_B)$ with $M^{-*}M^{-1}=\half(X'+Y')$.
\end{lemma}

\begin{proof}[\bf Proof.]

Applying Lemma \ref{L:defrel1} to both $A\prec B$ and $B\prec A$ yields $\ran \re(A)=\ran \re(B)$.
Moreover, we obtain that there exist operators $M,M'\in\cL(\fR_B)$ with $M^*M=\half(X+Y)$ and $M'^*M'=\half(X'+Y')$ such that $\re(A)^\half=M\re(B)^\half$ and $\re(B)^\half=M'\re(A)^\half$. Then $\re(A)^\half=MM'\re(A)^\half$ and $\re(B)^\half=M'M\re(B)^\half$. Hence $MM'=I=M'M$, which shows $M=M'^{-1}$ is invertible and $M^{-*}M^{-1}=M'^*M'=\half(X'+Y')$.
\end{proof}

The next theorem gives a characterization of this equivalence relation.

\begin{theorem}\label{T:eqrel}
Let $A,B\in\PR(\cH)$. Then $A\sim B$ if and only if one of the following equivalent statements holds:
\begin{itemize}
\item[(ERi)]
$ A-B=\re(A)^\half\wtilX \re(B)^\half \mbox{ for some } \wtilX\in\cL(\fR_B)=\cL(\fR_B,\fR_A)$;

\item[(ERii)]
$A^*+B=\re(A)^\half\wtilY \re(B)^\half \mbox{ for some }\wtilY\in\cL(\fR_B)=\cL(\fR_B,\fR_A)$;

\item[(ERiii)]
there exist $\wtil{r}>0$ with $\de B+(1-\de)A+\vep(A-B)\in\PR(\cH)$ for all $\de\in[0,1]$ and $\vep\in\BC$ with $|\vep|\leq \wtil{r}$;

\item[(ERiv)]
there exist $\wtil{r}>0$ with $\de B+(1-\de)A+\vep(A-B)\in\PR(\cH)$ for all $\de\in[0,1]$ and $\vep\in\BC$ with $|\vep|=\wtil{r}$.

\end{itemize}
Assume Assume $A\sim B$. Then the following statements hold:
\begin{itemize}
\item[(i)]
Let $X$ be as in (POi) and $Y$ as in (POii) and assume $B \prec A$ holds as in (POi) with $X$ replaced by $X'$ and as in (POii) with $Y$ replaced by $Y'$. Then $\wtilX$ and $\wtilY$ in (ERi) and (ERii) satisfy
\[
\|\wtilX\|\leq \|\re(X')+I\|\,\|X\|\ands
\|\wtilY\|\leq \|\re(Y')-I\|\,\|Y\|.
\]
Additional bounds on $\|\wtilX\|$ and $\|\wtilY\|$ are obtained by replacing the roles of $X$ and $X'$, respectively $Y$ and $Y'$.

\item[(ii)]
Let $\wtilX$ and $\wtilY$ be as in (ERi) and (ERii). Then $\wtilX+\wtilY$ is invertible and $(\wtilX+\wtilY) \re(B)^\half=2\re(A)^\half$. Moreover, the operators $X$, $X'$, $Y$ and $Y'$ in (i) satisfy
\begin{align*}
\max\{\|X\|,\|X'\|\}\leq\half(\|\wtilX\|+\sqrt{1+\|\wtilX\|})\|\wtilX\|,\quad
\max\{\|Y\|,\|Y'\|\}\leq \|\wtilY\|^2.
\end{align*}

\item[(iii)]
The numbers $\wtil{r}$ in (ERiii) and (ERiv) and $r$ and $r'$ in (POiii) and (POiv) for $A\prec B$ and $B\prec A$, respectively, can be taken such that $\wtil{r}=\min\{r,r'\}$.

\end{itemize}
\end{theorem}

The equivalence of (ERi)--(ERiii) is not as straightforward as for (POi)--(POiii), and we shall prove the equivalence indirectly, by showing that each of the statements is equivalent to $A\sim B$. To give some indication as to why the equivalence is not so straightforward, note that from (ERi) it is evident that $A\sim B$ implies $B^*\sim A^*$, while symmetry  ($A\sim B$ $\Rightarrow$ $B\sim A$) is not obvious from (ERi). On the other hand, the symmetry of $\sim$ follows immediately from (ERii), but here it is not directly clear that $A\sim B$ implies $A^* \sim B^*$.

Before proving Theorem \ref{T:eqrel}, we first prove a lemma which, in a more general setting, shows that the conclusion from Lemma \ref{L:defrel2} is also reached when $A\sim B$ is replaced by either (ERi) or (ERii). Here $\cran N$ denotes the closure of the range of the operator $N$. 

\begin{lemma}\label{L:DougApp}
Let $N_1\in\cL(\cK,\cK_1)$ and $N_2\in\cL(\cK,\cK_2)$ be Hilbert space operators. Assume there exists a $Z\in\cL(\cran N_1,\cran N_2)$ such that
\begin{equation}\label{GenCon}
N_1^*N_1\pm N_2^*N_2=\re(N_2^* Z N_1)
\end{equation}
with $\pm$ to be interpreted as either $+$ or $-$. Then $\kr N_1=\kr N_2$ and there exists an invertible operator $Q\in\cL(\cran N_1,\cran N_2)$ such that $Q N_1=N_2$. Moreover, if \eqref{GenCon} holds with $+$, then $\|Q^{-1}\|$ and $\|Q\|$ can be bounded by $\|Z\|$ and if \eqref{GenCon} holds with $-$, then $\|Q^{-1}\|$ and $\|Q\|$ can be bounded by $\half(\|Z\|+(1+\|Z\|)^\half)$.
\end{lemma}

\begin{proof}[\bf Proof.]
In both case, it suffices to show that there exist $\ga_1,\ga_2\geq 0$ such that
\begin{equation}\label{DougCon}
\|N_jx\|\leq \ga_j\|N_i x\|\quad (x\in\cK,\, i,j\in\{1,2\},\, i\not=j).
\end{equation}
Indeed, if this is the case then clearly $\kr N_i\subset\kr N_j$ and, again by Douglas' lemma, $N_j=Q_jN_j$ for an $Q_j\in \cL(\cran N_j, \cran N_i)$ with $\|Q_j\|\leq \ga_j$. Mimicking the proof of Lemma \ref{L:defrel2}, we obtain that $Q_1$ and $Q_2$ are invertible with $Q_1^{-1}=Q_2$. The bounds on $\|Q\|$ and $\|Q^{-1}\|$ then follow by showing that \eqref{DougCon} hold for appropriate choices of $\ga_j$.

First assume \eqref{GenCon} holds with $\pm$ replaced by $+$. In that case we have
\[
N_1^*N_1\leq \re(N_2^*ZN_1)\ands
N_2^*N_2\leq \re(N_1^*Z^*N_2),
\]
using $\re(N_2^*ZN_1)=\re((N_2^*ZN_1)^*)=\re(N_1^*Z^*N_2)$ in the last inequality. Then for any $x\in\cK$ and $i,j\in\{0,1\}$, $i\not=j$, we have
\begin{align*}
\|N_jx\|^2\leq \re(\inn{ZN_1x}{N_2x})\leq\|ZN_1x\|\|N_2x\|\leq \|Z\|\|N_ix\|\|N_jx\|.
\end{align*}
Thus \eqref{DougCon} holds with $\ga_j=\|Z\|$ for $j=1,2$. Hence we obtain $Q N_1=N_2$ for some invertible $Q\in\cL(\cran N_1,\cran N_2)$ with $\|Q\|$ and $\|Q^{-1}\|$ bounded by $\|Z\|$.

Now assume \eqref{GenCon} holds with $\pm$ replaced by $-$. Set $\ga=\half(\|Z\|+(1+\|Z\|)^\half)$. Then for $i,j\in\{0,1\}$, $i\not=j$, we have
\[
N_j^*N_j=\re(N_i^* Z_j N_j)+N_i^*N_i,
\]
with $Z_1=Z$ and $Z_2=-Z^*$. For each $x\in\cK$ this implies
\begin{align*}
\|N_jx\|^2=\re(\inn{Z_j N_jx}{N_ix})+\|N_ix\|^2
\leq \|Z_j\|\|N_jx\|\|N_ix\|+\|N_ix\|^2.
\end{align*}
The inclusion $\kr N_i\subset\kr N_j$ follows immediately from this inequality. In particular \eqref{DougCon} holds for $x\in \kr N_i$. Now assume $N_ix\not=0$ and set $\la_j=\|N_jx\|/\|N_ix\|$. Dividing by $\|N_ix\|^2$, we obtain that $\la^2\leq 1+\|Z_j\|\la=1+\|Z\|\la$. This inequality is satisfied for
\[
\half(\|Z\|-\sqrt{1+\|Z\|}) \leq \la \leq \half(\|Z\|+\sqrt{1+\|Z\|})=\ga.
\]
Thus $\la\leq\ga$ yields $\|N_j x\|\leq \ga\|N_i x\|$. Hence \eqref{DougCon} holds with $\ga_1=\ga_2=\ga$. Therefore, $Q N_1=N_2$ holds for some invertible $Q\in\cL(\cran N_1,\cran N_2)$ with $\|Q\|$ and $\|Q^{-1}\|$ bounded by $\ga$.
\end{proof}

\begin{proof}[\bf Proof of Theorem \ref{T:eqrel}.]
We first show that (ERiii) is equivalent to $A\sim B$, via (POiii) in both directions, and prove the relation between $r$, $r'$ and $\wtil{r}$ in (iii). The equivalence of (ERiii) and (ERiv), with the same value for $\wtil{r}$, goes along the same route as for (POiii) and (POiv). Clearly, (ERiii) implies (POiii) in both directions, with $r=r'=\wtil{r}$.
Now assume $A\sim B$ is established through (POiii) in both directions, with $r'$ for $B\prec A$. Fix $\vep\in\BC$ with $|\vep|\leq\wtil{r}:=\min\{r,r'\}$ and $\de\in[0,1]$.  Define $U=B+\vep(A-B)$ and $V=A-\vep(B-A)$. Then $U,V\in\PR(\cH)$ and, since $\PR(\cH)$ is convex, we have
\[
\de B+(1-\de)A+\vep(A-B)=\de U+(1-\de)V\in\PR(\cH).
\]
Hence (ERiii) holds.

Assume $A\sim B$. By Lemma \ref{L:defrel2}, $\fR_A=\fR_B$ and $\re(A)^\half=M\re(B)^{\half}$ for some invertible $M\in\cL(\fR_B)$. Clearly, (ERi) and (REii) then hold with $\wtilX= M^{-*}X$ and $\wtilY=M^{-*}Y$. Note further that
\[
\wtilX+\wtilY=M^{-*}(X+Y)=2M^{-*} M^*M=2M.
\]
Hence $(\wtilX+\wtilY)\re(B)^{\half}=2\re(A)^{\half}$. Since $M^{-*}M^{-1}=X'+Y'$, we have $\|M^{-*}\|^2=\|M^{-*}M^{-1}\|=\|X'+Y'\|$. Thus $\|\wtilX\|\leq\|X'+Y'\|^\half\|X\|$ and similarly $\|\wtilY\|\leq \|X'+Y'\|^\half\|Y\|$. The inequalities for $\|\wtilX\|$ and $\|\wtilY\|$ in (i) then follow from \eqref{XYrel}.

Next we employ Lemma \ref{L:DougApp} to show that (ERi) and (ERii) both imply  $\re(A)^\half=M\re(B)^{\half}$ for some invertible $M\in\cL(\fR_B)$, with appropriate bounds on $\|M\|$ and $\|M^{-1}\|$. Note that $A\sim B$ then follows immediately, since in both (ERi) and (ERii) one can then replace either $\re(A)^\half$ or $\re(B)^\half$ with the other.

Set $N_1=\re(A)^{\half}$ and $N_2=\re(B)^{\half}$. Taking real parts on both sides in (ERi) and (ERii), respectively, and using $\re(A^*)=\re(A)$ gives
\begin{align*}
N_1^*N_1-N_2^*N_2=\re(N_1^*\wtilX N_2)\ands
N_1^*N_1+N_2^*N_2=\re(N_1^*\wtilY N_2).
\end{align*}
Hence \eqref{GenCon} holds with $Z=\wtilX$ if $\pm=-$ and $Z=\wtilY$ if $\pm=+$. The result and the bounds in (ii) now follow immediately from Lemma \ref{L:DougApp}.
\end{proof}

We conclude this section with the analogue of Lemma 1.7 from \cite{tHsub}. The result follows from restricting Corollaries \ref{C:CaraStrict} and \ref{C:CaraInner} below to constant functions.

\begin{lemma} The following statements hold:
\begin{itemize}
\item[(i)] The set $\OPR(\cH)$ forms an equivalence class and $A\prec B$ holds for any $A\in\PR(\cH)$ and $B\in\OPR(\cH)$.

\item[(ii)] Any $B\in\PR(\cH)$ with $\re(B)=0$ forms an equivalence class by itself and $A\prec B$ implies $A=B$ for any $A\in\PR(\cH)$.

\end{itemize}
\end{lemma}

\section{Invariance under linear fractional transformations}\label{S:LFT}
\setcounter{equation}{0}

In this section we prove Theorems \ref{T:MainPrec} and \ref{T:MainEquiv}. Let $W$ be an invertible operator in $\cL(\cH\oplus\cH)$, for some Hilbert space $\cH$, and assume \eqref{Wcond} holds. Define $J$ and $\whatJ$ as in \eqref{JJhat} and $\wtilW$ as in \eqref{Wtil}.

\begin{lemma}\label{L:UseIneq}
Let $W\in \cL(\cH\oplus\cH)$ be invertible and assume \eqref{Wcond} is satisfied. Then for any $A,B\in\fD_W$ we have
\begin{equation}\label{UseIneq}
\begin{aligned}
\re(T_W[A])&\geq(W_{21}A+W_{22})^{-*}\re(A)(W_{21}A+W_{22})^{-1},\\
\re(T_W[A])&\geq(\wtilW_{11}+A\wtilW_{21})^{-1}\re(A) (\wtilW_{11}+A\wtilW_{21})^{-*},\\
T_W[A]-T_W[B]&=(\wtilW_{11}+A\wtilW_{21})^{-1}(A-B)(W_{22}+W_{21}B)^{-1}.
\end{aligned}
\end{equation}
In particular, for any $A\in\fD_W$  there exist contractions $M_A$ and $\wtilM_A$ such that
\begin{equation}\label{DougImplic}
\begin{aligned}
M_A\re(T_W[A])^\half&=\re(A)^\half(W_{22}+W_{21}B)^{-1},\\
\wtilM_A\re(T_W[A])^\half&=\re(A)^\half(\wtilW_{11}+A\wtilW_{21})^{-*}.
\end{aligned}
\end{equation}
\end{lemma}

The first inequality in fact holds without the invertibility of $W$ as well. Moreover, the fact that $T_W$ maps $\fD_W$ into $\PR(\cH)$, as claimed in the introduction, follows directly from the two inequalities in \eqref{UseIneq}.

\begin{proof}[\bf Proof of Lemma \ref{L:UseIneq}]
Using the two representations of $T_W$ given in \eqref{LFT} and \eqref{AltLFT} one easily verifies that
\begin{align*}
W\mat{c}{A\\I} &=\mat{c}{T_W[A]\\I}(W_{21}A+W_{22}),\\
\mat{cc}{I&A}\wtilW &=(\wtilW_{11}+A\wtilW_{21})\mat{cc}{I&T_W[A]}.
\end{align*}
Note that

Set $J_1=\sbm{0&I\\I&0}$ and $J_2=\sbm{I&0\\0&-I}$, both in $\cL(\cH\oplus\cH)$. Then
\begin{align*}
\wtilW \whatJ W=\whatJ,\quad W^*J W\leq J,\quad W J W^*\leq J.
\end{align*}
The identity follows directly from $\wtilW \whatJ=\whatJ W^{-1}$, the first inequality holds by assumption and the second inequality is a consequence of the first, cf., Lemma 2.3 in \cite{AD08}. The first inequality of \eqref{UseIneq} then follows from
\begin{align*}
&2(W_{21}A+W_{22})^*\re(T_W[A])(W_{21}A+W_{22})=\\
&\qquad=-(W_{21}A+W_{22})^*\mat{c}{T_W[A]\\I}^*J\mat{c}{T_W[A]\\I}(W_{21}A+W_{22})\\
&\qquad=-\mat{c}{A\\I}^*W^*J_1W\mat{c}{A\\I}
\leq\mat{c}{A\\I}^*J\mat{c}{A\\I}
=2\re(A).
\end{align*}
The second inequality is proved in a similar way, using the fact that $\wtilW$ is also $J$-contractive, which is a consequence of $\whatJ J \whatJ=-J$, details are left to the reader. The existence of contractions $M_A$ and $\wtilM_A$ satisfying \eqref{DougImplic} now follows directly from Douglas' lemma. Finally, the identity in \eqref{UseIneq} is a consequence of
\begin{align*}
&(\wtilW_{11}+B\wtilW_{21})(T_W[A]-T_W[B])(W_{21}A+W_{22})=\\
&\qquad =(\wtilW_{11}+B\wtilW_{21})\mat{cc}{I&T_W[B]}\whatJ\mat{c}{T_w[A]\\I}(W_{21}A+W_{22})\\
&\qquad =\mat{cc}{I&B}\wtilW \whatJ W\mat{c}{A\\I}=\mat{cc}{I&B}\whatJ\mat{c}{A\\I}=A-B.\qedhere
\end{align*}
\end{proof}

We now prove our second main result, Theorem \ref{T:MainEquiv}, in an extended form.

\begin{theorem}\label{T:MainEquivSpec}
Let $W\in \cL(\cH\oplus\cH)$ be invertible and assume \eqref{Wcond} is satisfied. Then $T_W$ preserves the equivalence relation $\sim$ on $\PR(\cH)$ restricted to $\fD_W$. More specifically, if $A\sim B$ for $A,B\in\fD_W$, say $A-B=\re(A)^\half \wtilX \re(B)^\half$ for $\wtilX\in \cL(\fR_B)$. Then
\begin{equation}\label{EquivPresId}
T_W[A]-T_W[B]=\re(T_W[A])^\half \wtilX_W \re(T_W[B])^\half\ \ \mbox{with}\ \ \wtilX_W=\wtilM_A^* \wtilX M_B.
\end{equation}
Here $M_B$ and $\wtilM_A$ are defined according to \eqref{DougImplic}. In particular, $\|\wtilX_W\|\leq \|\wtilX\|$.
\end{theorem}

\begin{proof}[\bf Proof.]
The identity \eqref{EquivPresId} follows after inserting $A-B=\re(A)^\half \wtilX \re(B)^\half$ into the identity in \eqref{UseIneq} and applying the identities in \eqref{DougImplic}, with $A$ replaced by $B$ in the first identity. Since $M_B$ and $\wtilM_A$ are contractions, we find $\|\wtilX_W\|\leq \|\wtilX\|$.
\end{proof}

Next we prove our first main result, Theorem \ref{T:MainPrec}, again in an extended form.

\begin{theorem}\label{T:MainPrecSpec}
Let $W\in \cL(\cH\oplus\cH)$ be invertible and assume \eqref{Wcond} is satisfied. Then $T_W$ preserves the pre-order $\prec$ on $\PR(\cH)$ restricted to $\fD_W$. More specifically, if $A\prec B$ for $A,B\in\fD_W$, say $A-B=\re(B)^\half \wtilX \re(B)^\half$ for $X\in \cL(\fR_B)$. Then
\begin{equation}\label{PrecPresId}
\begin{aligned}
&T_W[A]-T_W[B]=\re(T_W[B])^\half X_W \re(T_W[B])^\half,\\
&\mbox{with }\ \  X_W= \wtilM_B^*(I-X\re(B)^\half\wtilW_{21}(\wtilW_{11}+A\wtilW_{21})^{-1}\re(B)^\half)XM_B.
\end{aligned}
\end{equation}
Here $M_B$ and $\wtilM_A$ are defined according to \eqref{DougImplic}.
\end{theorem}

\begin{proof}[\bf Proof.]
Inserting $A-B=\re(B)^\half \wtilX \re(B)^\half$ into the identity in \eqref{UseIneq} and applying the first identity in \eqref{DougImplic}, with $A$ replaced by $B$, yields
\[
T_W[A]-T_W[B]=(\wtilW_{11}+A\wtilW_{21})^{-1} \re(B)^\half X_W \re(T_W[B])^\half.
\]
Hence, in order to complete the proof we have to show that
\begin{align*}
&(\wtilW_{12}+A\wtilW_{21})^{-1}\re(B)^\half=\\
&\qquad=\re(T_W[B])^\half \wtilM_B^*(I-X\re(B)^\half\wtilW_{21}(\wtilW_{11}+A\wtilW_{21})^{-1}\re(B)^\half).
\end{align*}
To see that this is the case first note that
\begin{align*}
&(\wtilW_{11}+B\wtilW_{21})^{-1}-(\wtilW_{11}+A\wtilW_{21})^{-1}=\\
&\qquad =(\wtilW_{11}+B\wtilW_{21})^{-1} [(\wtilW_{11}+A\wtilW_{21})-(\wtilW_{11}+B\wtilW_{21})] (\wtilW_{11}+A\wtilW_{21})^{-1}\\
&\qquad =(\wtilW_{11}+B\wtilW_{21})^{-1}(A-B)\wtilW_{21} (\wtilW_{11}+A\wtilW_{21})^{-1}\\
&\qquad =(\wtilW_{11}+B\wtilW_{21})^{-1}\re(B)^\half X\re(B)^\half\wtilW_{21} (\wtilW_{11}+A\wtilW_{21})^{-1}\\
&\qquad =\re(T_W[B])^\half\wtilM_B^*X\re(B)^\half\wtilW_{21} (\wtilW_{11}+A\wtilW_{21})^{-1}.
\end{align*}
Hence, we have
 \begin{align*}
&(\wtilW_{12}+A\wtilW_{21})^{-1}\re(B)^\half=\\
&\qquad =(\wtilW_{12}+B\wtilW_{21})^{-1}\re(B)^\half-\\ &\qquad\qquad+((\wtilW_{11}+B\wtilW_{21})^{-1}-(\wtilW_{11}+A\wtilW_{21})^{-1})\re(B)^\half\\
&\qquad =\re(T_W[B])^\half\wtilM_B^*-\\
&\qquad\qquad+\re(T_W[B])^\half\wtilM_B^*X\re(B)^\half\wtilW_{21} (\wtilW_{11}+A\wtilW_{21})^{-1}\re(B)^\half\\
&\qquad =\re(T_W[B])^\half\wtilM_B^*(I-
X\re(B)^\half\wtilW_{21} (\wtilW_{11}+A\wtilW_{21})^{-1}\re(B)^\half),
\end{align*}
as claimed.
\end{proof}

\section{An application to the Carath\'eodory class}
\label{S:Cara}\setcounter{equation}{0}

Throughout this section $\cU$ is a finite dimensional Hilbert space. We extend the pre-order $\prec$ and equivalence relation $\sim$ of Section \ref{S:pre-eq} to the Carath\'eodory class $\fC(\cU)$ and prove Theorems \ref{T:MainCaraEquiv} and \ref{T:MainCaraPrec}.

We start with some preliminaries. The operations $\re$ and $*$ are extended to $\fC(\cU)$ pointwise, i.e., for $F\in\fC(\cU)$ we define $\re(F)$ and $F^*$ by $\re(F)(\la)=\re(F(\la))$ and $F^*(\la)=F(\la)^*$, $\la\in\BD$. Recall that $\fR_F:=\fR_{F(\la)}$ is independent of the choice of $\la\in\BD$.

The following theorem provides several characterizations of the pre-order defined in \eqref{CaraPre} and the related equivalence relation.

\begin{theorem}\label{T:preceqC}
Let $F,G\in\fC(\cU)$. Then the relation $F  \precc G$ defined by one of the following four equivalent conditions:
\begin{itemize}
\item[(CPOi)]
$F-G= \re(G)^{\half}Q \re(G)^{\half}$  for a bounded $\cL(\fR_G)$-valued function $Q$ on $\BD$;

\item[(CPOii)]
$F^*\!+\!G\!=\! \re(G)^{\half}R\re(G)^{\half}$ for a bounded $\cL(\fR_G)$-valued function $R$ on $\BD$;

\item[(CPOiii)]
there exists an $s>0$ with $G+\vep(F-G)\in\fC(\cU)$ for all $\vep\in\BC$ with $|\vep|\leq s$;

\item[(CPOiv)]
there exists an $s>0$ with $G+\vep(F-G)\in\fC(\cU)$ for all $\vep\in\BC$ with $|\vep|= s$;

\end{itemize}
defines a pre-order relation on $\PR(\cH)$. Furthermore, we have $F\precc G$ and $G\precc F$ (denoted $F\simc G$) if and only if $\fR_F=\fR_G$ and one of the following equivalent statements holds:
\begin{itemize}
\item[(CERi)]
$F-G=\re(F)^\half\wtilQ\re(G)^\half$
for a bounded $\cL(\fR_G)$-valued function $\wtilQ$ on $\BD$;

\item[(CERii)]
$F^*\!+\!G\!=\!\re(F)^\half\wtilR\re(G)^\half$
for a bounded $\cL(\fR_G)$-valued function $\wtilR$ on $\BD$;

\item[(CERiii)]
there exists an $\wtil{s}>0$ with $\de G+(1-\de)F+\vep(F-G)\in\fC(\cU)$ for all $\de\in[0,1]$ and $\vep\in\BC$ with $|\vep|\leq \wtil{s}$;

\item[(CERiv)]
there exists an $\wtil{s}>0$ with $\de G+(1-\de)F+\vep(F-G)\in\fC(\cU)$ for all $\de\in[0,1]$ and $\vep\in\BC$ with $|\vep|=\wtil{s}$.
\end{itemize}
\end{theorem}

Similar relations exist between the supremum norms of the functions $Q$, $R$, $\wtilQ$ and $\wtilR$ and the numbers $s$ and $\wtil{s}$ as were derived for $X$, $Y$, $\wtilX$, $\wtilY$, $r$ and $\wtil{r}$  in Theorems \ref{T:MainPrecSpec} and \ref{T:MainEquivSpec}. However, we have no need for them in the sequel of the present paper.

\begin{proof}[\bf Proof of Theorem \ref{T:preceqC}]
The pointwise equivalences of (CPRi)--(CPRiv) and of (CERi)--(CERiv), i.e., with the equalities and inclusions at specified points of $\BD$ (and possibly different $r$ and $\wtil{r}$ at different points) follow immediately from the first parts of Theorems \ref{T:MainPrecSpec} and \ref{T:MainEquivSpec}, respectively. Hence we obtain the equivalence of (CPRi)--(CPRiv) and of (CERi)--(CERiv) without the boundedness constraint in (CPRi), (CPRii), (CERi) and (CERii) and with $s$ and $\wtil{s}$ in (CPRiii), (CPRiv), (CERiii) and (CERiv) possibly dependent of the point in $\BD$. The fact that we have equivalence with the boundedness conditions on $Q$, $R$, $\wtilQ$ and $\wtilR$ and with $s$ and $\wtil{s}$ independent of the point in $\BD$, follows directly from the inequalities in \eqref{XYrIneqs} and in items (i)--(iii) in Theorem \ref{T:MainEquivSpec}.
\end{proof}

As observed in the introduction, the interesting implications of $\precc$ appear on the boundary.

\begin{proposition}\label{P:limbehavior}
Let $F,G\in \fC(\cU)$ such that $F\precc G$, $u\in\cU$, and $t\mapsto \la_t$, $t\in(0,1]$ be a continuous curve in $\ov{\BD}$ with $\la_t\in\BD$ whenever $t\in(0,1)$. Assume $G(\la_1)$ exists in $\PR(\cU)$. Then $\lim_{t\uparrow 1} G(\la_t)u=0$ implies $\lim_{t\uparrow 1} F(\la_t)u=0$. In particular, if $\beta\in\BT$ and $\lim_{\la\to\beta}G(\la)u=0$ nontangentially (respectively unrestrictedly), then $\lim_{\la\to\beta}F(\la)u=0$ nontangentially (respectively unrestrictedly).
\end{proposition}

\begin{proof}[\bf Proof.]
First observe that for any $u\in\cU$ and $\la\in\BD$
\begin{align*}
\|\re(G)^\half(\la)u\|^2=|\re\inn{G(\la)u}{u}|\leq\|u\|\|G(\la)u\|.
\end{align*}
Now let $\wtilQ$ be as in (CPOi). Then
\begin{align*}
\|F(\la)u-G(\la)u\|
&\leq \|\re(G)^\half(\la) Q(\la)\re(G)^\half(\la)u\|\\
&\leq\|\re(G)^\half(\la)\| \|Q(\la)\|\|\re(G)^\half(\la)u\|\\
&\leq\|\re(G)^\half(\la)\| \|Q\|_\infty\sqrt{\|u\|\|G(\la)u\|}.
\end{align*}
Since $\|\re(G)^\half(\la_1)\|<\infty$, this inequality shows that $\lim_{t\to 1} G(\la_t)u=0$ implies $\lim_{t\to 1} F(\la_t)u=0$, as claimed.
\end{proof}

Due to the boundedness conditions in the various characterizations of $\precc$ and the fact that functions in $\fC_\circ(\cU)$ need not be bounded on $\BD$, the set of strict Carath\'eodory functions $\fC_\circ(\cU)$ is less well behaved with respect to the pre-order $\precc$ as is the case for strict Schur class functions in connection with the pre-order of \cite{tHsub}. We have to restrict to $\fC(\cU)\cap H^\infty(\cU)$.

\begin{corollary}\label{C:CaraStrict}
The set $\fC_\circ(\cU)\cap H^\infty(\cU)$ forms an equivalence class with respect to $\simc$ and $F\precc G$ holds for any $G\in \fC_\circ(\cU)\cap H^\infty(\cU)$ and $F\in\fC(\cU)\cap H^\infty(\cU)$.
\end{corollary}

\begin{proof}[\bf Proof.]
Note that if $G\in \fC(\cU)\cap H^\infty(\cU)$ and $F\in\fC(\cU)$, then $F\prec G$ implies  $F\in\fC(\cU)\cap H^\infty(\cU)$. Now assume $G\in \fC_\circ(\cU)\cap H^\infty(\cU)$. Then $\re(G)^\half$ is invertible on $\BD$ with $\la\mapsto \re(G)^{-\half}(\la)$ bounded on $\BD$. Hence $F-G=\re(G)^\half Q\re(G)^\half$ with $Q=\re(G)^{-\half}(F-G)\re(G)^{-\half}$ and $\|Q\|_\infty<\infty$ whenever $F\in\fC(\cU)\cap H^\infty(\cU)$. This shows $\fC_\circ(\cU)\cap H^\infty(\cU)$ is included in the equivalence class of any $G\in\fC_\circ(\cU)\cap H^\infty(\cU)$.

If $F\in\fC(\cU)\cap H^\infty(\cU)$, but $F\not\in\fC_\circ(\cU)$, then $F(e^{it})u=0$ for some $u\in\cU$ and $t\in\BR$ such that $F(e^{it})u$ can be defined through its nontangential limits at $e^{it}$. By Proposition \ref{P:limbehavior}, $G\precc F$ would imply $G(e^{it})u=0$, and thus $G\not\in\fC_\circ(\cU)$.
\end{proof}

The following result is a direct consequence of Proposition \ref{P:limbehavior} and the fact that functions in $\fC(\cU)$ are uniquely determined by their nontangential limits.

\begin{corollary}\label{C:CaraInner}
Any $G\in\fC(\cU)\cap H^\infty(\cU)$ with $\re(G)=0$ a.e.\ on $\BT$ forms an equivalence class by itself, and $F\precc G$ implies $F=G$ for any $F\in\fC(\cU)$.
\end{corollary}

Next we show that the functions that establish the relations $\precc$ and $\simc$ are continuous.

\begin{proposition}\label{P:contin}
Let $F,G\in\fC(\cU)$ such that $F\precc G$ (resp.\ $F\simc G$). Then the functions $R$ in (CPOi) and $Q$ in (CPii) (resp.\ $\wtilR$ in (CERi) and $\wtilQ$ in (CERii)) are continuous on $\BD$.
\end{proposition}

This result is a direct consequence of the following lemma.

\begin{lemma}\label{L:contin}
Let $\cU$, $\cV$ and $\cW$ be finite dimensional Hilbert spaces. Let $X$, $Y$ and $Z$ be functions on $\BD$ with values in $\cL(\cV,\cW)$, $\cL(\cV)$ and $\cL(\cU,\cV)$ and assume $\ran X(\la)^*=\cV=\ran Z(\la)$ for each $\la\in\BD$. Assume further that $Y$, $Z$ and $H:=XYZ$ are continuous on $\BD$ and $Y$ is bounded on $\BD$. Then $Y$ is continuous on $\BD$ as well.
\end{lemma} 

\begin{proof}[\bf Proof.]
We first show that $YZ$ is continuous on $\BD$. Let $\la_0\in\BD$. Since $\cV$ and $\cW$ be finite dimensional Hilbert spaces, $X$, $Z$ and $H$ are continuous with respect to any topology. Fix $u\in\cU$ and $w\in\cW$. Note that for any $\la\in\BD$ we have
\[
|\inn{Y(\la)Z(\la)u}{(X(\la)^*-X(\la_0)^*)w}|\leq\|Y(\la)\|\|Z(\la)u\|\|(X(\la)^*-X(\la_0)^*)w\|.
\]
Since $\la\mapsto X(\la)^*$ is continuous on $\BD$ and $H$ and $Y$ are bounded on a small enough neighborhood of $\la_0$, the above inequality yields
\[
\lim_{\la\to\la_0}|\inn{Y(\la)Z(\la)u}{(X(\la)^*-X(\la_0)^*)w}|=0.
\] 
Furthermore, we have 
\[
\lim_{\la\to\la_0}\inn{X(\la)Y(\la)Z(\la)u}{w}=\inn{X(\la_0)Y(\la_0)Z(\la_0)u}{w}
=\inn{Y(\la_0)Z(\la_0)u}{X(\la_0)^*w}
\]
and 
\begin{align*}
&\inn{X(\la)Y(\la)Z(\la)u}{w}=
\inn{Y(\la)Z(\la)u}{X(\la)^*w}=\\
&\qquad=\inn{Y(\la)Z(\la)u}{X(\la_0)^*w}+\inn{Y(\la)Z(\la)u}{(X(\la)^*-X(\la_0)^*)w}.
\end{align*}
This shows that 
\[
\lim_{\la\to\la_0}\inn{Y(\la)Z(\la)u}{X(\la_0)^*w}=
\inn{Y(\la_0)Z(\la_0)u}{X(\la_0)^*w}.
\]
Since $u\in\cU$ and $w\in \cW$ were chosen arbitrarily and $\ran X(\la_0)^*=\cV$, we obtain that $\lim_{\la\to\la_0} Y(\la)Z(\la)=Y(\la_0)Z(\la_0)$ for any $\la_0\in\BD$. Hence $YZ$ is continuous on $\BD$. Repeating the argument with $H=Z^*Y^*$, i.e., with $Z^*$ and $Y^*$ in place of $X$ and $Y$ and $Z$ identically equal to the identity operator in $\cV$ we obtain the continuity of $Y$. 
\end{proof}

Next we prove Theorems \ref{T:MainCaraEquiv} and \ref{T:MainCaraPrec}.

\begin{proof}[\bf Proof of Theorem \ref{T:MainCaraEquiv}.]
Let $F,G\in\fC(\cU)$ such that $F\simc G$, say $F-G=\re(F)^\half \wtilQ \re(G)^\half$ on $\BD$ with $\|\wtilQ\|_\infty<\infty$. Then we can apply Theorem \ref{T:MainEquivSpec} pointwise to all $\la\in\BD$, with $W=\Psi(\la)$, except for the few isolated points $\la\in\BD$ where $\Psi(\la)$ is not invertible. We then obtain $\|\wtilQ_\Psi(\la)\|\leq \|\wtilQ(\la)\|$ and
\begin{equation}\label{TFTGsim}
T_\Psi[F](\la)-T_\Psi[G](\la)=\re(T_\Psi[F])^\half(\la) \wtilQ_\Psi(\la) \re(T_\Psi[G])^\half(\la).
\end{equation}
Hence $T_\Psi[F](\la)\sim T_\Psi[G](\la)$ and we have $\|\wtilQ_\Psi(\la)\|\leq \|\wtilQ(\la)\|\leq \|\wtilQ\|_\infty<\infty$. 

Let $\la_0\in\BD$ such that $\Psi(\la_0)$ is invertible. Then for any of the isolated points $\la\in\BD$ where $\Psi(\la)$ is not invertible we have
\[
T_\Psi[F](\la)\sim T_\Psi[F](\la_0)\sim T_\Psi[G](\la_0)\sim T_\Psi[G](\la).
\]
Therefore, \eqref{TFTGsim} holds for all $\la\in\BD$. Lemma \ref{L:contin} shows that $\wtilQ_\Psi$ is continuous on $\BD$, and thus $\|\wtilQ_\Psi(\la)\|\leq \|\wtilQ\|_\infty$ also holds for the isolated points $\la$ where $\Psi(\la)$ is not invertible. Hence $\|\wtilQ_\Psi\|_\infty\leq \|\wtilQ\|_\infty<\infty$ and we obtain that $T_\Psi[F]\simc T_\Psi[G]$ holds via (CPOi).
\end{proof}

\begin{proof}[\bf Proof of Theorem \ref{T:MainCaraPrec}.]
The argumentation is similar to the proof of Theorem \ref{T:MainCaraEquiv}. However, in this case one needs
\[
(I-R\re(G)^\half\wtil{\Psi}_{21}(\wtil{\Psi}_{11}+F\wtil{\Psi}_{21})^{-1}\re(G)^\half)R
\]
to be bounded on $\BD$. The fact that $F,G\in\fD_\Psi^\circ$ implies that $(\wtil{\Psi}_{11}+F\wtil{\Psi}_{21})^{-1}$ and $\re(G)^\half$ exist and are bounded on $\BD$. Hence the result follows.
\end{proof}

\begin{example}\label{E:ce}
Take
\[
\Psi(\la)=\mat{cc}{1-\la & 1+\la \\ 1+\la & 1-\la}\quad (\la\in\BD).
\]
Then
\[
\Psi(\la)^* J_2 \Psi(\la)=\half\mat{cc}{1-|\la|^2 & 1+|\la|^2 \\ 1+|\la|^2 & 1-|\la|^2}\geq J_2 \quad(\la\in\BD)
\]
with equality on $\BT$. Moreover, $\det \Psi(\la)=-\la$, so that $\Psi$ is invertible on $\BD\backslash\{0\}$. Hence $T_\Psi$ defines a self-map of $\fC:=\fC(\BC)$.

Now take $F\equiv 0\in\fC$ and $G\equiv 1\in\fC_\circ$. Then $F\precc G$, by Corollary \ref{C:CaraStrict}. We have $\wtilG:=T_\Psi[G]\equiv 1\in\fC_\circ$ and 
\[
\wtilF(\la):=T_\Psi[F](\la)=\frac{1+\la}{1-\la}\mbox{ with } \re(\wtilF)(\la)=\frac{1-|\la|^2}{|1-\la|^2}\geq 0\quad (\la\in\BD).
\] 
Thus $\wtilF\not\in H^\infty$, and hence $\wtilF\not\precc \wtilG$ (see the proof of Corollary \ref{C:CaraStrict}), even though $\wtilG\in\fC_\circ$.
\end{example}


\end{document}